\documentclass[12pt]{article}

\usepackage{amsmath,amsthm}
\usepackage{amssymb}
\usepackage{enumitem}
\usepackage[T1]{fontenc}
\usepackage{url}

\newtheorem{theorem}{Theorem}[section]
\newtheorem{corollary}[theorem]{Corollary}
\newtheorem{lemma}[theorem]{Lemma}

\theoremstyle{definition}
\newtheorem{definition}[theorem]{Definition}

\newtheorem{example}[theorem]{Example}

\numberwithin{equation}{section}

\begin{document}

\title{Polymorphic Numbers with \\ Exponential Prefix}

\author{Samer Seraj\\
Existsforall Academy\\ 
Mississauga, Ontario, Canada\\
E-mail: samer\_seraj@outlook.com
}

\date{\today}

\maketitle

\renewcommand{\thefootnote}{}

\footnote{2020 \emph{Mathematics Subject Classification}: Primary 11A63, 11D09, 11A07, 26A06.}

\footnote{\emph{Key words and phrases}: Radix representation; Diophantine equations.}

\renewcommand{\thefootnote}{\arabic{footnote}}
\setcounter{footnote}{0}

\begin{abstract}
	We observe that the computation $5^2 = 25$ has the digital property of the result being equal to the exponent concatenated directly to the left of the base. The generalization to a Diophantine equation and inequality in number bases has been articulated previously, but a comprehensive answer was not available in the literature. We classify and largely parametrize the solutions. Tools that play key roles are the Newton-Raphson method, the arithmetic-geometric means inequality, Pell's equation, and Fermat's little theorem.
\end{abstract}
\hfill
\newpage

\section{Introduction}

The computation $5^2 =25$ has the interesting digital property where the decimal form is equal to the exponent appended directly to the left of the base. To explore when this phenomenon occurs in radix representations, we generalize it to a Diophantine equation and inequality. The question of what other solutions exist has been pondered in the past \cite{Lee}, but a comprehensive answer has not previously appeared.

\begin{definition}
As it is commonly defined, an \textit{automorphic number} in base $B\ge 2$ is a positive integer $x$ such that the rightmost digits of $x^2$ produce $x$ itself. Modifying the terminology of \cite{Numericana}, for an integer $n\ge 2$, an \textit{$n$-polymorphic number} $x$ in base $B$ satisfies the property that $x^n$ evaluates to having rightmost digits that produce $x$.
\end{definition}

\begin{example}
Examples of automorphic numbers in base $10$ are $x=25$ since $25^2 = 625$ ends in $25$, and $76$ since $76^2 = 5776$ ends in $76$. See \cite{10automorphic} for all automorphic numbers in base $10$.
\end{example}

Automorphic numbers are $n$-polymorphic numbers for all integers $n\ge 2$, but the converse is generally untrue.

\begin{example}
Examples of $3$-polymorphic numbers in base $10$ that are not automorphic numbers are $9$ since $9^3=729$ ends in $9$ but $9^2=81$ does not; and $24$ since $24^3 = 13824$ ends in $24$ but $24^2=576$ does not.
\end{example}

Our requirement is that $x$ is, not only $n$-polymorphic, but also that the remaining digits to the left of the digits of $x$ in $x^n$ produce the integer $n$. This is formalized in a definition as follows.

\begin{definition}\label{def:prefix-polymorphism}
The quadruple $(x,n,B,k)$ of positive integers with $B\ge 2$ will be said to be a \textit{prefix polymorphism} if
\begin{equation}\label{eqn:key-Diophantine}
    x^n = B^k n + x,
\end{equation}
where $k$ is the number of digits of $x$ in base $B$. The last condition implies that $k$ is a function of $x$ and $B$, and it is equivalent to saying that
\begin{equation}\label{eqn:key-inequality}
    B^{k-1} \le x < B^k.
\end{equation}
\end{definition}

Immediate consequences of Definition \ref{def:prefix-polymorphism} are, by contradiction:
\begin{itemize}
    \item $x\ge 2$, because otherwise $x=1$ in \eqref{eqn:key-Diophantine} implies that $1 = B^k n + 1$ or $B^k n =0$.
    \item $n\ge 2$, because otherwise $n=1$ in \eqref{eqn:key-Diophantine} implies that $x = B^k + x$ or $B^k =0$.
\end{itemize}

Our results for prefix polymorphisms $(x,n,B,k)$ are presented in the following order in subsequent sections:
\begin{enumerate}
    \item If there an integer solution $x=\alpha$ to \eqref{eqn:key-Diophantine}, then it must be $\alpha = \lfloor(B^k n)^{\frac{1}{n}}\rfloor+1$.
    \item It is impossible that $k\ge 3$. Moreover, $k=2$ and $n\ge 3$ cannot simultaneously hold.
    \item The non-empty classes are solved as follows.
    \begin{itemize}
        \item For $k=1$ and $n=2$, the exact parametrization is easily found.
        \item For $k=2$ and $n=2$, a full parametrization is given in terms of the solution to a Pell's equation. 
        \item For $k=1$ and $n\ge 3$, the solutions are biconditionally linked to weak Fermat pseudoprimes, and an explicit infinite family is given within this class.
    \end{itemize}
\end{enumerate}

\section{From Real to Integer Roots}

\begin{lemma}\label{lem:unique-root}
For the polynomial
$$f(x) = x^n -x - c,$$
where $n\ge 2$ is an integer and $c>0$ is real, there is exactly one positive real root $\alpha$.
\end{lemma}

\begin{proof}
We analyze the derivative of $f$ to determine the intervals of increase and decrease, and turning points. The derivative of $f$ leads to the unique positive critical point $\beta$:
$$f'(x) = nx^{n-1} - 1 =0 \implies \beta = x = \left(\frac{1}{n}\right)^{\frac{1}{n-1}},$$
though its negative $-\beta$ is also a critical point if $n$ is odd. The positive critical point is a turning point, specifically a local minimum, because $f'$ is negative directly to the left of $\beta$ and positive directly to the right of $\beta$.

Since $x^n -x$ has a real root at $x=0$ and the $-c$ shifts the curve down, $x^n - x -c$ instead gets a negative $y$-intercept. From there, as $x$ goes to the right, the analysis of $f'$ shows that $f$ decreases to the local minimum, and then increases to infinity. In that last segment, the $x$-axis will be crossed exactly once.

All of this is evident by looking at the graphs of curves $x^n - x$ for some fixed integers $n\ge 2$.
\end{proof}

\begin{lemma}[Arithmetic-geometric means inequality]
If $n$ is a positive integer and $a_1,a_2,\ldots,a_n$ are non-negative real numbers, then
$$n\cdot \sqrt[n]{a_1\cdot a_2\cdots a_n} \le a_1+a_2+\cdots+a_n.$$
Equality holds if and only if $a_1=a_2=\cdots=a_n$. \cite[p.~49]{Cvetkovski}
\end{lemma}

\begin{lemma}\label{lem:root-bounds}
If $n\ge 2$ is an integer and $c>1$ is real, then the unique positive real root $\alpha$ of $f(x) = x^n - x - c$ satisfies
$$c^{\frac{1}{n}} < \alpha < \frac{nc}{nc^{1-\frac{1}{n}}-1}.$$
\end{lemma}

\begin{proof}
In Lemma \ref{lem:unique-root}, the positive turning point was shown to be at $\beta = \left(\frac{1}{n}\right)^{\frac{1}{n-1}}$ for $n\ge 2$, so $0<\beta < 1.$ Then, for $x>1$ (where $1$ is a crude lower bound), $f$ is strictly increasing. Since $c>1$, it means $c$ is on this increasing segment, but it is to the left of the positive root $\alpha$ because
$$f(c^{\frac{1}{n}}) = (c^{\frac{1}{n}})^n - c^{\frac{1}{n}} -c = -c^{\frac{1}{n}}<0 = f(\alpha).$$ In other words, $c^{\frac{1}{n}} < \alpha$.

For the upper bound, we will use the Newton-Raphson method, starting with this $x_0 = c^{\frac{1}{n}}$ estimate of $\alpha$. We get the second point
$$x_1 = x_0 - \frac{f(x_0)}{f'(x_0)} = c^{\frac{1}{n}} - \frac{-c^{\frac{1}{n}}}{n(c^{\frac{1}{n}})^{n-1}-1} = \frac{nc}{nc^{1-\frac{1}{n}}-1} > 0.$$

We wish to show  that $x_1$ is to the right of $\alpha$. It suffices to prove that $f(x_1) > 0$, since $x_1$ is positive and the function is increasing to the right of $x_0$ with $f(x_0)<0$ and $f(\alpha) = 0$. This way, we will get the neat stacking $x_0 < \alpha < x_1$ and $f(x_0)<0$, $f(\alpha)=0$, and $0<f(x_1)$.

We want to prove
\begin{align*}
    0 < f(x_1) = f\left(\frac{nc}{nc^{1-\frac{1}{n}}-1}\right) &= \left(\frac{nc}{nc^{1-\frac{1}{n}}-1}\right)^n - \frac{nc}{nc^{1-\frac{1}{n}}-1} - c\\
    &= \frac{n^n c^n}{(nc^{1-\frac{1}{n}}-1)^n} - \frac{nc+nc^{2-\frac{1}{n}}-c}{nc^{1-\frac{1}{n}}-1}.
\end{align*}

Working backwards, we get
\begin{align*}
    \frac{nc+nc^{2-\frac{1}{n}}-c}{nc^{1-\frac{1}{n}}-1} &< \frac{n^n c^n}{(nc^{1-\frac{1}{n}}-1)^n}\\
    (nc+nc^{2-\frac{1}{n}}-c)(nc^{1-\frac{1}{n}}-1)^{n-1} & < n^n c^n\\
    (n-1+nc^{1-\frac{1}{n}})(nc^{1-\frac{1}{n}}-1)^{n-1} &< n^n c^{n-1}\\
    \sqrt[n]{(n-1+nc^{1-\frac{1}{n}})(nc^{1-\frac{1}{n}}-1)^{n-1}} & < nc^{1-\frac{1}{n}}.
\end{align*}

This is true because, by the arithmetic-geometric means inequality,
\begin{align*}
    &n\cdot \sqrt[n]{(n-1+nc^{1-\frac{1}{n}})(nc^{1-\frac{1}{n}}-1)^{n-1}} \\
    &\le (n-1+nc^{1-\frac{1}{n}})+(n-1)(nc^{1-\frac{1}{n}}-1)\\
    &= (n-1)(1+nc^{1-\frac{1}{n}}-1)+nc^{1-\frac{1}{n}}\\
    &= (n-1+1)\cdot nc^{1-\frac{1}{n}} = n^2c^{1-\frac{1}{n}}. 
\end{align*}

We just have to check that equality does not hold in order to make the inequality strict. By the equality criterion for the arithmetic-geometric means inequality, equality holds if and only if
$$n-1+nc^{1-\frac{1}{n}} = nc^{1-\frac{1}{n}}-1 \iff n = 0,$$
which is impossible because $n\ge 2$ for us.
\end{proof}

\begin{theorem}\label{thm:exact-root-c}
Suppose $n\ge 2$ is an integer and $c>1$ is real. If the unique positive real root $\alpha$ of $f(x) = x^n - x - c = 0$ is an integer, then $$\alpha = \lfloor c^{\frac{1}{n}} \rfloor + 1.$$
\end{theorem}

\begin{proof}
By Lemma \ref{lem:unique-root}, there is a unique positive real root $\alpha$ of $f$. By Lemma \ref{lem:root-bounds}, bounds on $\alpha$ are
$$c^{\frac{1}{n}} < \alpha < \frac{nc}{nc^{1-\frac{1}{n}}-1}.$$
If we can prove that is interval has length strictly less than $1$, then we will know that
\begin{equation}\label{eqn:alpha-bounds}
c^{\frac{1}{n}} < \alpha < \frac{nc}{nc^{1-\frac{1}{n}}-1} < c^{\frac{1}{n}} +1,
\end{equation}
since $c^{\frac{1}{n}} +1$ on the far right is exactly $1$ more than $c^{\frac{1}{n}}$ on the far left, instead of how $\frac{nc}{nc^{1-\frac{1}{n}}-1}$ is strictly less than $1$ more than $c^{\frac{1}{n}} +1$ on the far left. According to \cite[p.~69]{Knuth}, for any integer $\alpha$ and real $r$,
$$\lfloor r\rfloor = \alpha \iff r-1<\alpha\le r,$$
so we would get $$\alpha = \lfloor c^{\frac{1}{n}}+1\rfloor = \lfloor c^{\frac{1}{n}}\rfloor + 1.$$
So, we go backwards to prove
\begin{align*}
    1 &> \frac{nc}{nc^{1-\frac{1}{n}}-1} - c^{\frac{1}{n}} = \frac{c^{\frac{1}{n}}}{nc^{1-\frac{1}{n}}-1}\\
    c^{\frac{1}{n}} &< nc^{1-\frac{1}{n}}-1\\
    1 &< nc^{1-\frac{1}{n}} - c^{\frac{1}{n}} = c^{\frac{1}{n}} (nc^{1-\frac{2}{n}}-1),
\end{align*}
which is true because $c>1$ and $n\ge 2$.
\end{proof}

Of course, if the unique positive real root of $x^n - x - c$ is not an integer, then this formula does not work, though it can still be used for the purpose of eliminating cases through contradiction.

\begin{corollary}\label{cor:x-floor-formula}
If $(x,n,B,k)$ is a prefix polymorphism, then
\begin{equation}\label{eqn:x-floor-formula}
    x = \lfloor(B^k n)^{\frac{1}{n}}\rfloor + 1.
\end{equation}
\end{corollary}

\begin{proof}
Suppose $(x,n,B,k)$ is a prefix polymorphism. Since $x$ is the unique positive integer solution to $$x^n - x - B^k n = 0,$$
where $n\ge 2$ and $B^k n > 1$,
Theorem \ref{thm:exact-root-c} directly gives the stated formula \eqref{eqn:x-floor-formula} for $x$.
\end{proof}

\begin{example}
We can check that, if $B=10$, $n=2$, and $k=1$, then
$$x = \lfloor \sqrt{10^1 \cdot 2} \rfloor + 1 = 5,$$
as expected from the original example of $5^2 = 25$.
\end{example}

\section{Eliminating Higher Exponents}

We will see in this section that most pairs of the exponents $n$ and $k$ cannot produce a prefix polymorphism $(x,n,B,k)$.

\begin{lemma}\label{lem:knj-inequality}
If $(x,n,B,k)$ is a prefix polymorphism such that $n$ has $j$ digits in base $B$, then
$$(k-1)(n-1) \le j.$$
\end{lemma}

\begin{proof}
Suppose $(x,n,B,k)$ is a prefix polymorphism. Since $x$ has $k$ digits in base $B$, we know that $B^{k-1} \le x$. To make this inequality strict, and therefore stronger, we will eliminate the possibility that $x = B^{k-1}$. So, suppose for contradiction, that $x = B^{k-1}$. Then,
\begin{align*}
    0 &= x^n - x - B^k n = (B^{k-1})^n - B^{k-1} - B^k n\\
    0 &= B^{(k-1)(n-1)} - 1 - Bn.
\end{align*}
We know that $n\ge 2$ by definition. If $k\ge 2$ as well, then we find that $B\mid 1$ or $B=1$, which is impossible. Looking at the case where $k=1$, we similarly get the contradiction $1=Bn +1$ or $B=0$. Thus, $x\ne B^{k-1}$, allowing us to state that
\begin{equation}\label{x-1-lower-bound}
B^{k-1} < x \implies B^{k-1} \le x-1 \implies B^{(k-1)n} \le (x-1)^n.
\end{equation}
By \eqref{eqn:alpha-bounds},
$$(B^k n)^{\frac{1}{n}} < x < (B^k n)^{\frac{1}{n}}  + 1,$$
where the right inequality gives
\begin{equation}\label{x-1-upper-bound}
    (x-1)^n < B^k n.
\end{equation}
Combining \eqref{x-1-lower-bound} and \eqref{x-1-upper-bound} yields
$$B^{(k-1)n} \le (x-1)^n < B^k n \implies B^{kn-n-k} < n.$$
Since $n$ has $j$ digits in base $B$, we get
\begin{align*}
    B^{kn-n-k} < n < B^j &\implies kn - n - k < j\\
    &\implies (k-1)(n-1) = kn - k - n + 1 \le j,
\end{align*}
as desired.
\end{proof}

\begin{lemma}\label{lem:power-linear-ineq}
For all integers $j\ge 2$, $$2^j \ge j$$ with equality holding if and only if $j=2$. Secondly, for all integers $n\ge 3$, $$2^{n-1} -1 \ge n$$ with equality holding if and only if $n=3$.
\end{lemma}

\begin{proof}
These can be proven by comparing the growth rates of each side of each inequality, or by using induction.
\end{proof}

\begin{lemma}\label{lem:n-equals-j}
If $n\ge 2$ is an integer and $j$ is the number of digits of $n$ in base $B\ge 2$, then $n\ge j$. Equality holds, as in $n=j$, if and only if $n=j=B=2$.
\end{lemma}

\begin{proof}
Since $n$ has $j$ digits in base $B$, we know that $n \ge B^{j-1}$. Since $B\ge 2$, we get $$n \ge B^{j-1} \ge 2^{j-1}.$$ By Lemma \ref{lem:power-linear-ineq}, $2^{j-1} \ge j$ for $j\ge 2$ with equality if and only if $j=2$. So,
$$n \ge 2^{j-1} \ge j.$$ Lastly, $j=n$ if and only if
$$n \ge B^{j-1} \ge 2^{j-1} \ge j = n.$$
Since the upper and lower bounds are equal, all intermediate expressions are also equal to $n$. Thus, $B^{j-1} = 2^{j-1} = n \ge 2$ if and only if $B=2$, and we already know that $2^{j-1} = j=n$ if and only if $n=j=2$.
\end{proof}

\begin{theorem}\label{thm:k-1-2}
If $(x,n,B,k)$ is a prefix polymorphism, then $k=1$ or $k=2$. In other words, there are no prefix polymorphisms with $k\ge 3$.
\end{theorem}

\begin{proof}
Let $j$ be the number of digits of $n$ in base $B$. If we can prove that $j\ne n$, then Lemma \ref{lem:n-equals-j} would say that $n>j$ or $n-1 \ge j$. Then, Lemma \ref{lem:knj-inequality} would give
$$(n-1)(k-1) \le j \le n-1 \implies k-1 \le 1 \implies k \le 2,$$
where division by $n-1$ would be allowed because $n\ge 2$.

So, we need to only eliminate the possibility that $j=n$. By Lemma \ref{lem:n-equals-j}, if $j=n$, then $n=j=B=2$, so
$$0 = x^n - x - B^k n = x^2 - x - 8,$$
which has no integer solutions.
\end{proof}

\begin{theorem}\label{thm:n-above-3-is-k-1}
Suppose $(x,n,B,k)$ is a prefix polymorphism. If $n\ge 3$, then $k=1$.
\end{theorem}

\begin{proof}
Assume $n\ge 3$ in the prefix polymorphism. Suppose, for the sake of contradiction, that $k=2$. Then,
\begin{align*}
    B^2 n = B^k n &= x^n - x = x (x^{n-1}-1)\\
    &= x(x-1)(x^{n-2} + x^{n-3} + \cdots + x +1)\\
    &\ge x(x-1)(2^{n-2} + 2^{n-3} + \cdots + 2 +1)\\
    &= x(x-1)(2^{n-1}-1)
\end{align*}
Now, $2^{n-1}-1 \ge n$  by Lemma \ref{lem:power-linear-ineq} because $n\ge 3$. To make the inequality strict, we will show that $n\ne 3$. Suppose, for the sake of contradiction, that $n=3$. Then, Lemma \ref{lem:knj-inequality} and Lemma \ref{lem:n-equals-j} give
$$2 = (2-1)(3-1)= (k-1)(n-1) \le j.$$
Since $j$ is the number of digits of $n=3$ in base $B$, the only way that $j\ge 2$ is if $B=2$ or $B=3$. We can use Corollary \ref{cor:x-floor-formula} to check that
\begin{align*}
    B=2 &\implies x = \lfloor(2^2 \cdot 3)^{\frac{1}{3}}\rfloor + 1 =3 \\
    &\implies x^n - x - B^k n = 3^3 - 3 - 2^2\cdot 3 = 12 \ne 0,\\
    B=3 &\implies x = \lfloor(3^2 \cdot 3)^{\frac{1}{3}}\rfloor + 1 =4 \\
    &\implies x^n - x - B^k n = 4^3 - 4 - 3^2\cdot 3 = 33 \ne 0.
\end{align*}

So, $n\ne 3$, implying $n\ge 4$. This gives the strict inequality
\begin{align*}
B^k n \ge x(x-1)(2^{n-1}-1) > x(x-1)n &\implies B^k n > x(x-1)n\\
    &\implies B^k > x(x-1)\\
    &\implies 4B^k +1 > (2x-1)^2\\
    &\implies 4B^k \ge (2x-1)^2\\
    &\implies 2B^{\frac{k}{2}} \ge 2x-1.
\end{align*}
Suppose, for contradiction, that $k=2$. Then, since $B$ and $x$ are integers,
$$2B \ge 2x -1 \implies B + \frac{1}{2} \ge x \implies B \ge x.$$
The only way for to have $k=2$ digits in base $B$ while satisfying $x\le B$ is if $x=B$. Then, we would get
\begin{align*}
    0 &= x^n - x - B^k n = B^n - B - B^2 n\\
    0 &= B^{n-1} - 1 - Bn,
\end{align*}
which would imply that $B\mid 1$ or $B=1$. This is a contradiction, so $k\ne 2$, which automatically means $k=1$ by Theorem \ref{thm:k-1-2}.
\end{proof}

\begin{corollary}\label{cor:solution-classes}
Prefix polymorphisms $(x,n,B,k)$ can be classified into three disjoint families:
$$\begin{array}{l|c|c}
    \text{Class} & k & n \\ \hline
    \text{Triangular} & k=1 & n=2\\
    \text{Pell} & k=2 & n=2\\
    \text{Fermat} & k=1 & n\ge 3
\end{array}$$
\end{corollary}

\begin{proof}
By Theorem \ref{thm:k-1-2}, $k=1$ or $k=2$. By definition, $n \ge 2$, so $n=2$ or $n \ge 3$. By Theorem \ref{thm:n-above-3-is-k-1}, $k=2$ and $n\ge 3$ cannot simultaneously hold. This proves that the stated classes are the only ones, and it is clear that they are disjoint.
\end{proof}

The name of the classes in Corollary \ref{cor:solution-classes} will make sense when we parame-trize or otherwise explore the families in the next section.

\section{Classifying Solutions}

\begin{theorem}[Triangular class]\label{thm:triangular-class}
All prefix polymorphisms $(x,n,B,k)$ such that $n=2$ and $k=1$ are parametrized by
$$(x,n,B,k) = \left(t,2,\frac{t(t-1)}{2},1\right),$$
ranging over all integers $t\ge 4$.
\end{theorem}

\begin{proof}
Suppose $n=2$ and $k=1$. Then
$$0 = x^n - x - B^k n = x^2 - x - 2B \implies B = \frac{x(x-1)}{2}.$$
Taking $t=x$ completes the parametrization, though we need to check that $x=t$ has $k=1$ digit in base $B = \frac{t(t-1)}{2}$. This requires
\begin{align*}
    x < B &\iff t < \frac{t(t-1)}{2} \iff 2t < t^2 - t\\
    &\iff 2 < t - 1 \iff 3 < t \iff 4 \le t.
\end{align*}
To be explicit, we can check that $x=t=2$ leads to $B=1$ which contradicts $x<B$ (and that $B\ge 2$), and that $x=t=3$ leads to $B=3$ which contradicts $x<B$ as well. Therefore, $t\ge 4$ is necessary and sufficient.
\end{proof}

\begin{example}
A simple case of Theorem \ref{thm:triangular-class} takes $t=4$ to get $(x,n,B,k) = (4,2,6,1)$, and we can check that
$$4_6^2 = 4_{10}^2 = 16_{10} = (24)_{6}.$$
\end{example}

\begin{theorem}[Pell class]\label{thm:Pell-class}
All prefix polymorphisms $(x,n,B,k)$ with $n=2$ and $k=2$ are parametrized by
$$(x,n,B,k) = \left(\frac{1+x_t}{2},2,y_t,2\right),$$
where $(x_t,y_t)$ for each integer $t\ge 2$ is given by the expansion of
\begin{equation}\label{eqn:Pell-parametrization}
    x_t + y_t \sqrt{8} = (3+\sqrt{8})^t.
\end{equation}
\end{theorem}

\begin{proof}
Suppose $n=2$ and $k=2$. Then, by the quadratic formula,
\begin{align*}
    x^n - x - B^k n = 0 &\implies x^2 - x - 2B^2 = 0\\
    &\implies x = \frac{1+\sqrt{1+8B^2}}{2},
\end{align*}
where the root with the negative square root is extraneous since it yields $x<0$.

It is sufficient and necessary that $1+8B^2$ is a square because that square would be odd, making $\frac{1+\sqrt{1+8B^2}}{2}$ a positive integer. The solutions to $1+8B^2 = z^2$ for positive integers $(z,B)$ are those of the Pell's equation $$z^2 - 8B^2 = 1.$$
Since the fundamental solution is $(z,B) = (3,1)$, this has exactly the solution set described by \eqref{eqn:Pell-parametrization}, where we take $t\ge 2$ to avoid $B=1$ in the $t=1$ scenario. For a reference to solving Pell's equation, see \cite[p.~36]{Titu}.

Moreover, we can prove backwards that $x$ has $k=2$ digits in base $B$:
\begin{align*}
    B^{k-1} \le x < B^k &\iff B \le x < B^2\\
    &\iff B \le \frac{1+\sqrt{1+8B^2}}{2} < B^2\\
    &\iff (2B-1)^2 \le 8B^2 \le (2B^2-1)^2\\
    &\iff 4B^2 - 4B + 1 \le 1+8B^2 < 4B^4 - 4B^2 +1.
\end{align*}
The left inequality is equivalent to
$$4B^2 - 4B + 1 \le 1+8B^2 \iff 4B^2 - 4B \le 8B^2 \iff 0 \le 4B^2 + 4B,$$
and the right inequality is equivalent to
$$1+8B^2 < 4B^4 - 4B^2 +1 \iff 12 B^2 < 4B^4 \iff 3 < B^2,$$
both of which are true, the latter because $B\ge 2$.
\end{proof}

\begin{example}
The simplest case of Theorem \ref{thm:Pell-class} is $(3+\sqrt{8})^2 = 17 + 6\sqrt{8}$, so we get $$(x,n,B,k) = (9,2,6,2),$$ which gives the example $13_6^2 = 9_{10}^2 = 81_{10} = 213_6$.
\end{example}

\begin{theorem}[Fermat class]\label{thm:Fermat-class}
If $(x,n,B,k)$ is a prefix polymorphism such that $n\ge 3$, then $$x^n \equiv x \pmod{n},$$
which is a phenomenon where $n$ is called a weak Fermat pseudoprime to base $x$; as a consequence, $\frac{x^n - x}{n}$ is an integer. Conversely, all quadruples 
\begin{equation}\label{eqn:Fermat-parametrization}
    (x,n,B,k) = \left(t,n,\frac{t^n -t}{n},1\right),
\end{equation}
such that $t\ge 2$, $n\ge 3$, and $t^n \equiv t \pmod{n}$ form a prefix polymorphism, except when $(x,n,B,k) = (2,3,2,1)$.
\end{theorem}

\begin{proof}
The first direction is easy, since any prefix polymorphism $(x,n,B,k)$ satisfies 
$$x^n = x + B^k n \implies x^n \equiv x \pmod{n}.$$
Conversely, assume \eqref{eqn:Fermat-parametrization} and suppose $x^n \equiv x \pmod{n}$. Then, there exists an integer $q$ such that
$$x^n = x + qn \implies q = \frac{x^n - x}{n}.$$
Since we require $x^n = x + B^1 n$ in the prefix polymorphism being constructed, we have to set $B=q$. The remaining requirements are the inequalities
$$B^{k-1}\le x < B^k \iff 1\le t < B,$$
which would also automatically also lead to $1<B \iff 2\le B$. Equivalently,
$$t<B \iff t < \frac{t^n - t}{n} \iff n < t^{n-1} -1.$$
Since we require that $t=x\ge 2$, it suffices to prove the tighter inequality
$$2^{n-1} - 1 > n.$$
Following Lemma \ref{lem:power-linear-ineq}, we know that $2^{n-1} - 1 \ge n$, with equality holding if and only if $n=3$, so we just have to study the $n=3$ case separately.

Suppose $n=3$, and recall that we are assuming $k=1$. Then,
$$0 = x^n -x - B^k n = x^3 -x - 3B \implies B = \frac{x^3 -x}{3}.$$
Then, we get the equivalent inequalities
\begin{align*}
x<B &\iff x < \frac{x^3 - x}{3} \iff 4x < x^3 \\
&\iff 4 < x^2 \iff 2 < x \iff 3\le x.
\end{align*}
So, there is potential trouble only for $x<3$. To investigate, if $x=2$, then $B = \frac{2^3 - 2}{3} = 2 = x,$ contradicting $x<B$. So, we need to exclude only $(x,n,B,k) = (2,3,2,1)$.
\end{proof}

\begin{corollary}\label{cor:prime-family}
An infinite family of prefix polymorphisms in the Fermat class is given by
$$(x,n,B,k) = \left(t,p,\frac{t^p-t}{p},1\right),$$
for integers $t\ge 2$ and odd primes $p$, excluding $(x,n,B,k) = (2,3,2,1)$. This gives an explicit family of prefix polymorphisms with unbounded $n$.
\end{corollary}

\begin{proof}
By Fermat's little theorem \cite[p.~78]{Hardy}, if $x$ is an integer and $p$ is a prime, then
$$x^p \equiv x \pmod{p}.$$
The result is now a special family within the Fermat class described in Theorem \ref{thm:Fermat-class}.
\end{proof}

\begin{example}
A simple case of Theorem \ref{thm:Fermat-class} takes $t=2$ and $n=5$, giving $B= \frac{2^5 - 2}{5} = 6$, so
$$(x,n,B,k) = (2,5,6,1).$$
We can check that $2_6^5 = 2_{10}^5 = 32_{10} = 52_6$.
\end{example}

\end{document}